\documentclass[12pt,a4paper]{article}
\usepackage{amssymb,amsmath,amsthm}
\usepackage{graphicx}
\usepackage{hyperref}

\newcommand\cA{{\mathcal A}}

\newcommand\cF{{\mathcal F}}

\newcommand\cG{{\mathcal G}}
\newcommand\cH{{\mathcal H}}

\newcommand\cN{{\mathcal N}}

\newcommand\cQ{{\mathcal Q}}

\newcommand\cT{{\mathcal T}}

\theoremstyle{plain}
\newtheorem{theorem}{Theorem}[section]
\newtheorem{lemma}[theorem]{Lemma}
\newtheorem{corollary}[theorem]{Corollary}

\newtheorem{proposition}[theorem]{Proposition}

\theoremstyle{definition}

\newtheorem{defn}[theorem]{Definition}
\newtheorem{con}[theorem]{Construction}

\newcommand\cref[1]{Corollary~\ref{cor:#1}}

\title{$t$-wise Berge and $t$-heavy hypergraphs}
\author{D\'aniel Gerbner, D\'aniel T. Nagy, Bal\'azs Patk\'os, M\'at\'e Vizer \\
\medskip
\small Alfr\'ed R\'enyi Institute of Mathematics, Hungarian Academy of Sciences\\
\small P.O.B. 127, Budapest H-1364, Hungary.\\
\medskip
\small \texttt{\{gerbner,nagydani,patkos\}@renyi.hu, vizermate@gmail.com}
\medskip}

\begin{document}

\maketitle
\begin{abstract}

In many proofs concerning extremal parameters of Berge hypergraphs one starts with analyzing that part of that shadow graph which is contained in many hyperedges. Capturing this phenomenon we introduce two new types of hypergraphs. A hypergraph $\cH$ is a $t$-heavy copy of a graph $F$ if there is a copy of $F$ on its vertex set such that each edge of $F$ is contained in at least $t$ hyperedges of $\cH$. $\cH$ is a $t$-wise Berge copy of $F$ if additionally for distinct edges of $F$ those $t$ hyperedges are distinct.

We extend known upper bounds on the Tur\'an number of Berge hypergraphs to the $t$-wise Berge hypergraphs case. We asymptotically determine the Tur\'an number of $t$-heavy and $t$-wise Berge copies of long paths and cycles and exactly determine the Tur\'an number of $t$-heavy and $t$-wise Berge copies of cliques.

    In the case of 3-uniform hypergraphs, we consider the problem  in more details and obtain additional results.
\end{abstract}
\section{Introduction}

Problems in extremal (hyper)graph theory deal with determining those $n$-vertex (hyper)graphs with a prescribed property that are optimal ``in some sense". Tur\'an type problems ask for the largest number of (hyper)edges that an $n$-vertex (hyper)graph $\cH$ can contain if it does not contain a forbidden sub(hyper)graph $\cF$.
The asymptotics of the Tur\'an number is determined by the celebrated result of Erd\H os, Stone and Simonovits if $\cF$ is a non-bipartite graph and there are famous solved and open problems about the Tur\'an number of bipartite graphs.

Much less is known in the hypergraph case. General surveys on the topic are that of Keevash \cite{kee} and Chapter 5 of the book by Gerbner and Patk\'os \cite{gp}. Apart from sporadic results, researchers tried to define hypergraph classes the corresponding Tur\'an type problems are approachable.
Extending the way Berge defined hypergraph cycles, Gerbner and Palmer \cite{gp1} introduced the following.

\begin{defn} We say that a hypergraph $\cF$ is a \textit{Berge copy} of a graph $F$, if there exists an injection $i:V(F)\rightarrow V(\cF)$ and a bijection $b:E(F)\rightarrow E(\cF)$ such that for any edge $(xy)=e\in E(F)$ we have $\{i(x),i(y)\}\subseteq b(e)$.
\end{defn}

In other words, we can obtain a Berge copy of a graph $F$ by extending every edge $e$ of $F$ to a larger hyperedge such that all the hyperedges obtained in this way are \textit{distinct}. There are many other ways how one can create hypergraphs from graphs using different sets of rules. Mubayi and Verstra\"ete \cite{mubver} survey expansions, a more general enumeration of extremal results concerning so-called graph-based hypergraphs can be found in \cite{gp}.

For Berge hypergraphs, extremal problems have been widely studied, see e.g. \cite{gmp,gmv,gp1,gp2,gykl,gyl,gymstv}.
In this paper we introduce and study generalizations of the Berge hypergraph concept.

\vspace{2mm}
The motivation for introducing these notions is the following: in many proofs concerning Berge hypergraphs one starts to analyze the part of that shadow graph which is contained in many hyperedges.
The \textit{shadow graph} of a hypergraph $\cH$ has the same vertex set as $\cH$ and those pairs of vertices form an edge in the shadow graph that are contained in at least one hyperedge of $\cH$.
In this paper we mainly focus on the edges of the shadow graph with 'large multiplicities'. More precisely, we say that an edge of the shadow graph is \textit{$t$-heavy}, if it is contained in at least $t$ hyperedges.

\begin{defn} For an integer $t \ge 1$  and a graph $F$ we say that a hypergraph $\cF$ is a \textit{$t$-heavy} copy of $F$ if there exists an injection $i:V(F)\rightarrow V(\cF)$ and a function $h:E(F)\rightarrow \binom{E(\cF)}{t}$ such that for any edge $(xy)=e\in E(F)$ we have $\{i(x),i(y)\}\subseteq \cap_{A\in h(e)}A$.
We denote the family of $t$-heavy copies of $F$ by $\mathbb{H}_tF$, and the family of $r$-uniform $t$-heavy copies of $F$ by $\mathbb{H}^r_tF$.
\end{defn}

Equivalently, $\cF$ is a \textit{$t$-heavy} copy of $F$ if the subgraph of the shadow graph consisting of $t$-heavy edges contains a copy of $F$.
With a little abuse of notation we will say that the \textit{essence} $F_\cF$ of a copy $\cF$ of $F$ is the graph isomorphic to $F$ with vertex set $\{i(x):x\in V(F)\}$ and edge set $\{((i(x),i(y)):(xy)\in E(F)\}$. Note that $F_\cF$ might depend on the injection $i$ as well.

\vspace{2mm}

We also introduce another notion that is a clear generalization of Berge hypergraphs.

\begin{defn} We say that a hypergraph $\cF$ is a \textit{$t$-wise Berge} copy of a graph $F$ if $|E(\cF)|=t|E(F)|$ and there exists an injection $i:V(F)\rightarrow V(\cF)$ and function $h:E(F)\rightarrow \binom{E(\cF)}{t}$ such that
\begin{itemize}
    \item
    for any pair $e,e'$ of different edges in $F$, we have $h(e)\cap h(e')=\emptyset$, and
    \item
    for any edge $(xy)=e\in E(F)$ we have $\{i(x),i(y)\}\subseteq \cap_{A\in h(e)}A$.
\end{itemize} We denote the family of $t$-wise Berge copies of $F$ by $\mathbb{B}_tF$, and the family of $r$-uniform $t$-wise Berge copies of $F$ by $\mathbb{B}^r_tF$. Note that the case $t=1$ recovers the original Berge copies of $F$.

\end{defn}

If $\cF$ is a $t$-wise Berge copy of $F$, then let the \textit{essence} $F_\cF$ of $\cF$ to be again the graph isomorphic to $F$ with vertex set $\{i(x):x\in V(F)\}$ and edge set $\{((i(x),i(y)):(xy)\in E(F)\}$.

The easiest way to imagine a Berge copy $\cF$ of $F$ is to replace all edges $e$ of the essence by $t$ hyperedges containing $e$ such that all $t|E(F)|$ newly introduced hyperedges are distinct.

Note that the difference between $t$-wise Berge and $t$-heavy copies $\cF$ of $F$ is that when obtaining them from their essence by replacing edges by hyperedges, the newly introduced hyperedges must be all distinct in case of $t$-wise Berge copies, while for $t$-heavy copies there might be overlaps between the sets of $t$ hyperedges containing distinct edges of $F$.

Obviously a $t$-wise Berge copy of $F$ is also a $t$-heavy copy of $F$.
Note that a hypergraph can be a $t$-wise Berge copy or a $t$-heavy copy of many different graphs.

\begin{defn}
$\bullet$ For a family $\mathbb{F}$ of hypergraphs, let us define its Tur\'an number $ex(n,\mathbb{F})$ as the largest number of edges in an $\mathbb{F}$-free hypergraph on $n$ vertices.

$\bullet$ For a family $\mathbb{F}$ of ($r$-uniform) hypergraphs, let us define its Tur\'an number $ex_r(n,\mathbb{F})$ as the largest number of edges in an $\mathbb{F}$-free $r$-uniform hypergraph on $n$ vertices.
\end{defn}

By definition, we have $\mathbb{B}^r_tF \subseteq \mathbb{H}^r_tF$, so

\begin{proposition} For any integers $r,t,n$ and graph $F$ we have $$ex_r(n,\mathbb{H}_t^rF)\le ex_r(n,\mathbb{B}_t^rF).$$
\end{proposition}

Hence we will often state upper bounds only for $t$-wise Berge copies and lower bounds for $t$-heavy copies of graphs.

Before stating our results, we need to introduce a related area. Let $F$, $G$ and $H$ be graphs, then $\cN(H,G)$ denotes the number of copies of $H$ in $G$, and $ex(n,H,F):=\max\{\cN(H,G): G\, \text{is an $F$-free graph on $n$ vertices}\}$. These so-called generalized Tur\'an problems were studied for several pairs of graphs. Their systematic study was initiated by Alon and Shikhelman \cite{as}, for further results see e.g.~\cite{ggymv} and references therein.
A connection between Berge hypergraphs and generalized Tur\'an problems was pointed out by Gerbner and Palmer \cite{gp2}. A connection to a colored version of generalized Tur\'an problems was established by Gerbner, Methuku and Palmer \cite{gmp} and also by F\"uredi, Kostochka and Luo \cite{fkl} in an equivalent form. We say that a graph is blue-red if all of its edges are colored either blue or red. We denote the graph spanned by the blue edges by $G_{blue}$ and the graph spanned by the red edges by $G_{red}$.

\begin{theorem}\label{celebrated2} For any graph $F$ there is an $F$-free blue-red graph $G$ on $n$ vertices such that $$ex_r(n,\mathbb{B}_t^rF)\le \cN(K_r,G_{blue})+t|E(G_{red})|+(t-1)\left(\binom{n}{2}-|E(G)|\right).$$
\end{theorem}

Using this theorem, we can prove the following.

\begin{theorem}\label{pathcycle} For any integers $r\ge 3$, $t\ge 2$ and $k\ge 2(t-1)(r-2)+2$ we have
$$ex_r(n,\mathbb{H}_t^rP_k), ex_r(n,\mathbb{B}_t^rP_k)=(t-1-o(1))\binom{n}{2}$$
and
$$ex_r(n,\mathbb{H}_t^rC_{k-1}), ex_r(n,\mathbb{B}_t^rC_{k-1})=(t-1-o(1))\binom{n}{2}.$$

\end{theorem}
Let us now consider cliques. First we construct $\mathbb{H}_t^rK_k$-free $r$-uniform hypergraphs if $k>r$. Let $\cT^r(n,k-1)$ be the complete balanced $(k-1)$-partite $r$-uniform hypergraph on $n$ vertices (here balanced means that the cardinality of two parts differ by at most one).
 We take $\cT^r(n,k-1)$ and for every pair of vertices $u$ and $v$ from the same part, we take $t-1$ different $(r-2)$-sets that intersect $r-2$ other parts (this is doable only if $t$ is small enough compared to $r$ or $n$). We add the union of these $(r-2)$-sets and $u$ and $v$ as hyperedges to obtain the hypergraph $\cQ^r_t(n,k-1)$.

\begin{theorem}\label{klikk} Let $k>r>2$, $t\ge 1$ and $n$ be large enough. Then \[ex_r(n,\mathbb{B}_t^rK_k)=ex_r(n,\mathbb{H}_t^rK_k)=|\cQ^r_t(n,k-1)|.\]

\end{theorem}

\bigskip

The structure of the paper is as follows. In Section 2, we prove some general results including Theorem \ref{celebrated2} and Theorem \ref{pathcycle}. Section 3 contains results concerning cliques including the proof of Theorem \ref{klikk}. In Section 4, we gather several result for 3-uniform hypergraphs. In the case $t$ equals $2$, we determine the asymptotics of $ex_3(n,\mathbb{B}^3_2F)$ for almost all graphs $F$.

\bigskip

Let us finish the introduction by adding some notation. Hyperedges and sets of size $r$ will be often referred to as $r$-hyperedges and
$r$-sets. Our notation for order of magnitude and asymptotics is standard. We will always think of the uniformity $r$ and the multiplicity $t$ as fixed and the number $n$ of vertices growing to infinity. Therefore when a result of the form $f(n)=O(g(n))$ is stated, it means that there exists a constant $C$ such that $f(n)\le Cg(n)$ holds, but $C$ might depend on other parameters like $r$, $t$ or the forbidden graph $F$.

\section{General results}

We start with some simple observations that show how $t$-heavy and $t$-wise Berge hypergraphs are connected to better understood graph-based hypergraphs and also give a quadratic lower bound on $ex(n,\mathbb{H}^r_t(F))$ for any graph $F$ if $t\ge 2$.

The \textit{$r$-uniform expansion} $F^{+r}$ of a graph $F$ is obtained by adding $(r-2)|E(F)|$ new vertices to $V(F)$ to obtain $V(F^{+r})$ and replacing each edge $e \in E(F)$ by an $r$-hyperedge $E_e$ such that $E_e\cap V(F)=e$ and $(E_e\cap E_{e'})\setminus V(F)=\emptyset$ for any pair of different edges $e,e' \in E(F)$. Tur\'an problems for expansions have also been widely investigated, see \cite{mubver} for a survey. A hypergraph $\cH$ with the property that for any distinct pair $H,H'\in E(\cH)$ of hyperedges we have $|H\cap H'|\le 1$ is called a \textit{linear hypergraph}. Extremal problems have been investigated for this class of hypergraphs, see example \cite{fo,f,t}.

\begin{proposition}\label{trivi}

Let $F=(V(F),E(F))$ be an arbitrary graph.

(i) If $t\ge 2$, and $F$ has an edge, then a linear hypergraph does not contain any $t$-heavy copy of $F$.

(ii) Let $m:= \min\{\binom{r}{2}, |E(F)|\}$ and $t':=\lfloor \frac{t}{m}\rfloor$, then every $r$-uniform $t$-heavy copy of $F$ contains a $\mathbb{B}^r_{t'}F$.

(iii) If $t\ge V(F)+|E(F)|-2$, then every 3-uniform $t$-heavy copy of $F$ contains $F^{+3}$.

\end{proposition}

\begin{proof}
Observe first that (i) is trivial, as no edge is contained in two hyperedges.

To prove (ii), we build an auxiliary bipartite graph $G$. One of its parts $A$  corresponds to the hyperedges of $\cH$, the other part $B$ corresponds to the edges of the essence $F_\cH$, and $a\in A$ is connected to $b\in B$ if the corresponding hyperedge contains the corresponding edge. Observe that vertices of $A$ have degree  at most $\min\{\binom{r}{2}, |E(F)|\}$ and vertices of $B$ have degree at least $t$ in $G$. Let us build another graph $G'$ with parts $A$ and $B'$ by replacing each vertex in $B$ by $t'$ copies of it, connected to the same vertices of $A$. Then vertices of $A$ have degree at most $\min\{\binom{r}{2}, |E(F)|\}t'\le t$  and vertices of $B'$ have degree at least $t$ in $G'$.
Hence there is a matching covering $B'$ in $G'$ using Hall's theorem. This matching gives a $t'$-wise Berge copy of $F$.

To prove (iii), consider a $t$-heavy copy $\cH$ of $F$. We go through the edges of the essence $F_\cH$ in an arbitrary order, and pick a hyperedge containing the edge in such a way that the resulting hypergraph is $F^{+3}$. For any edge $\{u,v\}$, we pick a hyperedge $\{u,v,x\}$ such that $x$ is not contained in previously chosen hyperedges.
Note that $F^{+3}$ has $|V(F)|+|E(F)|$ vertices. Thus even at the last edge $uv$, we want to avoid only $|V(F)|+|E(F)|-1$ vertices. Therefore one of the $t$ hyperedges containing $uv$ has an unused vertex, and we are done.
\end{proof}

\begin{corollary}\label{coroll} Let $F$ be an arbitrary graph.

(i)\label{quad} For any graph $F$ with at least one edge and any integer $t\ge 2$ we have $$ex_r(n,\mathbb{H}_t^rF)\ge (t-1-o(1))\frac{\binom{n}{2}}{\binom{r}{2}}.$$

(ii) For any graph $F$ and integer $t\ge \min\{\binom{r}{2}, |E(F)|\}$, we have $$ex_r(n,\mathbb{H}_t^rF)\ge ex_r(n,\mathbb{B}_1^rF).$$
\end{corollary}

\begin{proof} To prove (i), we will consider an $r$-uniform hypergraph where every edge is contained in at most $t-1$ hyperedges (hence it is obviously $\mathbb{H}_t^rF$-free). If $n$ is divisible by $t-1$, there exists a balanced incomplete block design (i.e. an $r$-uniform hypergraph where every edge is contained in exactly $t-1$ hyperedges) by the celebrated result of Wilson \cite{wilson}. It is easy to see that such a hypergraph has cardinality $(t-1)\binom{n}{2}/\binom{r}{2}$. If $n$ is not divisible by $t-1$, we take the largest $n'<n$ such that $t-1$ divides $n'$, and a balanced incomplete block design on $n'$ vertices. It has $(t-1-o(1))\binom{n}{2}/\binom{r}{2}$ hyperedges.

(ii) follows from (ii) of Proposition \ref{trivi}.

\end{proof}

One of our main goal in this paper is to extend extremal results about Berge hypergraphs to $t$-wise Berge and $t$-heavy hypergraphs.

\begin{proposition}\label{tberge} For $r \ge 3$ and $t \ge 2$ we have

\vspace{2mm}

(i) $ex(n,\mathbb{B}_tF)\le ex(n,\mathbb{B}_{t-1}F)+\binom{n}{2}$, and

\vspace{1mm}

(ii) $ex_r(n,\mathbb{B}^r_tF)\le ex_r(n,\mathbb{B}^r_{t-1}F)+\binom{n}{2}$.

\end{proposition}

\begin{proof} The proof of (i) and (ii) are similar, we only present the proof of (i).

Let $\cH$ be a $\mathbb{B}_tF$-free hypergraph. Let $S \subseteq \binom{V(H)}{2}$
be a maximal set such that there exists a bijection $g : S \rightarrow E(H)$ such that $e \subseteq g(e)$ for all $e \in S$. Let $\cH'$ be the hypergraph obtained from $\cH$ by deleting the hyperedges in $g(S)$. Then $\cH'$ is $\mathbb{B}_{t-1}F$-free since otherwise, $E(F_{\cH'}) \subseteq S$ by maximality of $S$, and thus this $(t-1)$-wise Berge copy can be extended to a $t$-wise Berge copy. Clearly, $|\cH|=|\cH'|+|S|\le |\cH'|+\binom{n}{2}$.
\end{proof}

The very first result about general Berge hypergraphs is a theorem of Gerbner and Palmer \cite{gp1} stating that for any graph $F$, if $\cH$ is a $\mathbb{B}_{1}F$-free hypergraph and every hyperedge of $\cH$ has size at least $|V(F)|$, then $\sum_{H\in \cH}|H|=O(n^2)$. Here and in the rest of the paper we consider fixed graphs $F$ and fixed integers $t$, while $n$ goes to infinity. It was extended by English, Gerbner, Methuku and Palmer \cite{sgmp} showing under the same conditions that $\sum_{H\in \cH}|H|^2=O(n^2)$. Here we extend it further, showing that the same holds for any $t$ if $\cH$ does not contain any $B \in \mathbb{B}_{t}F$.

\begin{proposition} For any graph $F$, if $\cH$ does not contain $t$-wise Berge copies of $F$ and every hyperedge of $\cH$ has size at least $|V(F)|$, then $\sum_{H\in \cH}|H|^2=O(n^2)$.

\end{proposition}

\begin{proof} Let $G$ be the graph on $V(\cH)$ that consists of the edges that are contained in at most $t|E(F)|-1$ hyperedges of $\cH$. Furthermore, let $G'$ be the complement graph. Then we claim $G'$ is $F$-free. Indeed, otherwise we could go through the edges of $F$ and greedily choose $t$ different, previously unused hyperedges containing them.

Let us fix a real $\alpha$ such that $1>\alpha>\frac{1}{\chi(F)-1}$ if $\chi(F)\ge 3$ and $0<\alpha<1$ if $\chi(F)=2$. According to the Erd\H os-Stone-Simonovits theorem there exists an $n_0\ge |V(F)|$ such that any $F$-free graph on $n\ge n_0$ vertices contains at most $(1-\alpha)\binom{n}{2}$ edges. Let $\beta:=\min\{\alpha,\frac{1}{\binom{n_0}{2}}\}$.

Consider a hyperedge $H$. As $G'$ is $F$-free, so is $G'[H]$, and as $|H|\ge |V(F)|$, there must be at least one edge missing from $G'[H]$. Hence there are at least $\beta\binom{|H|}{2}$ edges in $G[H]$.

Finally, let us consider the pairs $(e,H)$ when $e$ is an edge of $G$ and $H$ is a hyperedge of $\cH$ that contains $e$. On one hand, there are at most $\binom{n}{2}$ edges in $G$, and for each such edge there are at most $t|E(F)|-1$ hyperedges in $\cH$ containing it. On the other hand, the number of such pairs is at least $\sum_{H\in \cH}\beta\binom{|H|}{2}$, hence $\sum_{H\in \cH}\beta\binom{|H|}{2}\le (t|E(F)|-1)\binom{n}{2}$, finishing the proof.
\end{proof}

Gerbner and Palmer \cite{gp2} proved the following very useful lemma for Berge hypergraphs.

\begin{lemma}[Gerbner, Palmer, \cite{gp2}]\label{celeb} For any graph $F$ we have $$ex(n,K_r,F)\le ex_r(n,\mathbb{B}^r_1F)\le ex(n,K_r,F)+ex(n,F).$$

\end{lemma}

The following strengthened version was obtained independently by Gerbner, Methuku and Palmer \cite{gmp} and by F\"uredi, Kostochka and Luo \cite{fkl}.
We state the result in the form used in \cite{gmp}.

\begin{theorem}[Gerbner, Methuku, Palmer \cite{gmp}]\label{celeb2}
For any graph $F$ there is an $F$-free blue-red graph $G$ on $n$ vertices such that $$ex_r(n,\mathbb{B}^r_1F)\le \cN(K_r,G_{blue})+|E(G_{red})|.$$

\end{theorem}

Here we extend both of the above statements for $t$-heavy and $t$-wise Berge hypergraphs.
Both proofs are based on the same basic ideas as the proofs for the Berge versions. Note that the  proposition below follows from Theorem \ref{celebrated2} and also follows from combining Lemma \ref{celeb} and Proposition \ref{tberge}, but we add the proof for sake of completeness and as a warm-up for the next proof.

\begin{proposition}\label{celebrated} For any graph $F$ and integer $t$, we have $$ex(n,K_r,F)\le ex_r(n,\mathbb{H}_t^rF)\le ex_r(n,\mathbb{B}_t^rF)\le ex(n,K_r,F)+ex(n,F)+(t-1)\binom{n}{2}.$$

\end{proposition}

\begin{proof}
For the lower bound we take an $F$-free graph on $n$ vertices and place a hyperedge on the vertex set of every $r$-clique. The resulting hypergraph obviously does not contain any $t$-heavy copies of $F$.

To prove the upper bound, let us consider an $r$-uniform hypergraph $\cH$ without any $t$-wise Berge copies of $F$ and fix an arbitrary order $H_1,H_2,\dots$ of the hyperedges of $\cH$. We go through the hyperedges in that order, and for each $H_i$ we try to pick a subedge of $H_i$. We always choose a subedge that has been picked less than $t$ times for hyperedges $H_j$ with $j<i$. If it is impossible, because all the subedges of $H_i$ have been picked $t$ times earlier, we mark the $r$-clique on the vertices of the hyperedge. After this process ends, let $G_i$ be the graph of the edges picked exactly $i$ times. Observe that $G_t$ is $F$-free and every marked $r$-clique is in $G_t$. Let $x$ be the number of $r$-cliques marked, then we have $x\le ex(n,K_r,F)$. We claim that $|\cH|=x+\sum_{i=0}^ti|E(G_i)|$. Indeed, for every hyperedge we either marked a clique or moved an edge from $G_i$ to $G_{i+1}$ for some $i\le t-1$. As $\sum_{i=0}^ti|E(G_i)|\le |E(G_t)|+(t-1)\binom{n}{2}$, we are done.
\end{proof}

Recall that if $t\ge 2$, we have a quadratic lower bound on $ex_r(n,\mathbb{H}^r_tF)$ by (i) of Proposition \ref{quad}. The difference between the upper and lower bound in Proposition \ref{celebrated} is at most quadratic, thus we know the exact order of magnitude of $ex_r(n,\mathbb{H}_t^rF)$ and $ex_r(n,\mathbb{B}_t^rF)$ in case we know the order of magnitude $ex(n,K_r,F)$. More precisely, we have the following.

\begin{corollary} If $t\ge 2$, then for any graph $F$ the following four quantities have the same order of magnitude

 $$ex_r(n,\mathbb{H}_t^rF),~ex_r(n,\mathbb{B}_t^rF),~\max\{ex(n,K_r,F),n^2\},~\max\{ex_r(n,\mathbb{B}_1^rF),n^2\}.$$

\end{corollary}

Moreover, in case $ex(n,K_r,F)$ (or equivalently $ex_r(n,\mathbb{B}_1^rF)$) is super-quadratic and its asymptotics is known, then we also know the asymptotics of $ex_r(n,\mathbb{H}_t^rF)$ and $ex_r(n,\mathbb{B}_t^rF)$.

\newtheorem*{celebrated2}{Theorem \ref{celebrated2}}
\begin{celebrated2}
For any graph $F$ there is an $F$-free blue-red graph $G$ on $n$ vertices such that $$ex_r(n,\mathbb{B}_t^rF)\le \cN(K_r,G_{blue})+t|E(G_{red})|+(t-1)\left(\binom{n}{2}-|E(G)|\right).$$
\end{celebrated2}

\begin{proof}
Let $\cH$ be a $\mathbb{B}_t^rF$-free $r$-uniform hypergraph on $n$ vertices. We define an auxiliary bipartite graph $X$. Part $A$ of $X$ consists of the hyperedges in $\cH$ and part $B$ of $X$ consists of the edges of the shadow graph of $\cH$. A vertex $a\in A$ and a vertex $b\in B$ are connected if $a\supset b$.

Let us take a subgraph $X'$ of $X$ that has the largest number of edges among subgraphs of $X$ satisfying the following properties: every vertex in $A$ has degree at most 1 and every vertex in $B$ has degree at most $t$ in $X'$. We denote by $d(v)$ the degree of a vertex in $X'$. Let $A'=\{a\in A: d(a)=1\}$ and $B'=\{b\in B: d(b)=t\}$. Observe that there is no edge between $A\setminus A'$ and $B\setminus B'$, as that could be added to $X'$. Also observe that the subgraph $G$ of the shadow graph consisting of the edges corresponding to vertices in $B'$ is $F$-free. Indeed, as shown by $X'$, those edges are each contained in at least $t$ hyperedges of $\cH$, and as $d(a)\le 1$ for all $a\in A$, these hyperedges are distinct for edges corresponding to distinct vertices in $B'$. Therefore our aim is to red-blue color $G$, i.e. the vertices of $B'$. To this end we will partition $A$ and $B$ into several parts, but first we introduce some terminology.

Pairs $a,\in A, b\in B$ will be called $X$-edge, $X'$-edge, $(X-X')$-edge if they form an edge in $X$, $X'$, in $X$ but not in $X'$, respectively. We call a path in $X$ \textit{alternating} if it alternates between $X'$-edges and $(X-X')$-edges. Observe that there is no alternating path $P$ starting in $A\setminus A'$ and ending in $B\setminus B'$, as then we could consider the subgraph $X''$ of $X$ obtained from $X'$ by removing the $X'$-edges of $P$ and adding the $(X-X')$-edges of $P$. This way the degree of the first and last vertex increases, and the degrees of the other vertices do not change, hence we obtained a larger subgraph satisfying the degree conditions; a contradiction.

During the proof several times we will build an alternating path connecting two vertices $a$ and $b$ by putting together multiple alternating paths: one from $a$ to $b'$, another from $b'$ to $a'$ and a third from $a'$ to $b$. This builds an alternating walk instead of a path, as a vertex could be used multiple times. However, one can easily see that such an alternating walk contains an alternating path from $a$ to $b$.

Let $A_1$ denote the vertices in $A'$ that can be reached from $B\setminus B'$ with an alternating path that starts and ends with an $(X-X')$-edge. Similarly let $B_1$ denote the vertices in $B'$ that can be reached from $A\setminus A'$ with an alternating path. Note that here the starting edge is by definition an $(X-X')$-edge, as there is no $X'$-edge incident with a vertex of $A\setminus A'$.

Let $A_2$ denote the vertices in $A'$ that are connected to $B_1$ with an $X'$-edge, and let $B_2$ denote the vertices in $B'$ that are connected to $A_1$ with an alternating path starting (and ending) with an $X'$-edge. Observe that a vertex $a\in A_1\cap A_2$ would be reached with an alternating path from $B\setminus B'$ ending with an $(X-X')$-edge  (as $a\in A_1$), and for the $X'$-edge $ab$, we have $b\in B_1$ (as $a\in A_2$). Then we could continue the alternating path starting with the edge $ab$. The vertex $b$ is reached from $A\setminus A'$ with an alternating path, and the last edge of that path is an $(X-X')$-edge. Hence we can continue our alternating path from $b$ to $A\setminus A'$. Thus we found an alternating walk, hence an alternating path also from $B\setminus B'$ to $A\setminus A'$, a contradiction.
Therefore, we have $A_1\cap A_2=\emptyset$.

Similarly $B_1\cap B_2=\emptyset$. Indeed, assume $b\in B_1\cap B_2$. The alternating path from $A\setminus A'$ to $b$, that exists as $b\in B_1$, has to start with an $(X-X')$-edge, thus it also ends with an $(X-X')$-edge. Then from $b$, as it belongs to $B_2$, we could continue with an alternating path to $a\in A_1$ starting and ending with an $X'$-edge. Then we can continue with an alternating path starting with an $(X-X')$-edge to $B\setminus B'$ from $a$. These three paths form an alternating walk from $a\setminus A'$ to $B\setminus B'$, leading to a contradiction again.

Let $A_3=A'\setminus (A_1\cup A_2)$ and $B_3=B'\setminus (B_1\cup B_2)$. By the above, $A_1,A_2,A_3$ form a partition of $A'$ and $B_1,B_2,B_3$ form a partition of $B'$. Let us color the edges of $G$  corresponding to vertices in $B_1\cup B_3$ blue and let the edges corresponding to vertices in $B_2$ be red. To get the desired inequality for this red-blue coloring of $G$ we need some further observations and refining the partition of $A'$.

Let us further partition $A_3$ into two parts. $A_4\subseteq A_3$ is the set of vertices that are connected by the $X'$-edge to $B_2$ or $B\setminus B'$ and $A_5=A_3\setminus A_4$. This way there is no $X$-edge between $A_5$ and $B_2$, as an $X'$-edge is forbidden by definition, and an $(X-X')$-edge would extend an alternating path from $B\setminus B'$ to $A_1$, then to $B_2$, and then to a vertex $a\in A_5$, but in this case $a$ would be in $A_1$, a contradiction. Clearly, we have $|\cH|=|A|=(|A\setminus A'|+|A_2|+|A_5|)+(|A_1|+|A_4|)$. Therefore it is enough to prove $|A\setminus A'|+|A_2|+|A_5|\le \cN(K_r,G_{blue})$ and $|A_1|+|A_4|\le t|B_2|+(t-1)|B\setminus B'|$.

To see $|A\setminus A'|+|A_2|+|A_5|\le \cN(K_r,G_{blue})$, we need to prove that all $X$-edges from $A\setminus A',A_2$, and $A_5$ go to $B_1$ or $B_3$, as then for every hyperedge $a\in (A\setminus A')\cup A_2\cup A_5$, the edges in the shadow of $a$ form an $r$-clique in $G_{blue}$. From $B_2\cup (B\setminus B')$ every $(X-X')$-edge goes to $A_1$, as it extends an alternating path starting from $B\setminus B'$. On the other hand, by definition of $A'$, there are no $X'$-edges incident any $a\in A\setminus A'$. Also, by definition of $A_2$, the only $X'$-edge incident to an $a\in A_2$ goes to $B_1$, and by definition of $A_4$, the only $X'$-edge adjacent to $a\in A_5$ goes to $B\setminus (B_2\cup (B\setminus B'))=B_1\cup B_3$.

 To see $|A_1|+|A_4|\le t|B_2|+(t-1)|B\setminus B'|$, observe first that, by definition of $B_2$, from $A_1$ every $X'$-edge goes to $B_2$. Also, by definition of $A_4$, from $A_4$ the $X'$-edges go to $B_2$ and $B\setminus B'$. Therefore the number of $X'$-edges incident to $A_1\cup A_4$ is exactly $|A_1|+|A_4|$ (every $a\in A'$ is incident to exactly one $X'$-edge), and their number is at most $t|B_2|+(t-1)|B\setminus B'|$ (the number of $X'$-edges incident to $b\in B'$ is $t$ and to $b'\in B\setminus B'$ is at most $t-1$).
\end{proof}

For $t$-heavy hypergraphs we can obtain a stronger statement with a simpler proof.

\begin{proposition} For any graph $F$ there is an $F$-free graph $G$ such that $$ex(n,\mathbb{H}_t^rF)\le \cN(K_r,G)+(t-1)\left(\binom{n}{2}-|E(G)|\right).$$

\end{proposition}

\begin{proof}

Let us consider an $\mathbb{H}_t^rF$-free hypergraph $\cF$, and let $G$ be the graph consisting of the edges in the shadow graph that are contained in at least $t$ hyperedges. Then $G$ is obviously $F$-free. The number of hyperedges with their shadow in $G$ is at most $ \cN(K_r,G)$. Every other hyperedge contains an edge not in $G$, but such an edge is counted at most $t-1$ times, finishing the proof.
\end{proof}

 Let us turn our attention to constructions.

 \begin{con}\label{con1}
 Let us assume there exists
 a $(t-1)$-regular $(r-1)$-uniform $\mathbb{H}_1^{r-1}F$-free hypergraph $\cF$ on $k$ vertices. Then we take $\lfloor n/k\rfloor$ disjoint $k$-sets $A_1,\dots, A_{\lfloor n/k\rfloor}$, and a copy of $\cF$, $\cF_i$ on every $A_i$. Let $\cH$ be the $r$-uniform hypergraph containing the following hyperedges. For every $i<j$, we take every $r$-set that consists of an element of $A_i$ and a hyperedge from $\cF_j$.
 \end{con}

 Let us consider $u\in A_i$ and $v\in A_j$. They are contained in exactly $t-1$ hyperedges of $\cH$. This means that if $\cH$ contains a $t$-heavy copy of any graph, the edges of the essence of this copy are inside the $A_i$'s. Therefore,  if $F$ is connected, then $\cH$ does not contain a $t$-heavy copy of $F$. The number of hyperedges in $\cH$ is $(1-o(1))(t-1)\binom{n}{2}/(r-1)$, as there are $(1-o(1))\binom{n}{2}$ edges between $A_i$'s, each of them is contained in $t-1$ hyperedges, and every hyperedge is counted $r-1$ times this way. This implies the following.

 \begin{proposition}\label{glb} Let $F$ be a connected graph with $|V(F)|$ large enough compared to $t$ and $r$. Then $$ex_r(n,\mathbb{H}_t^rF)\ge (1-o(1))\frac{t-1}{r-1}\binom{n}{2}.$$

\end{proposition}

\begin{proof}
Let $\cF$ be a $(t-1)$-regular $(r-1)$-uniform hypergraph on the smallest number of vertices possible and let us write $|V(\cF)|=k$. Then it is $F$-free for every $F$ with more than $k$ vertices. So the $r$-uniform hypergraph $\cH$ given by Construction \ref{con1} based on $\cF$ and the calculation in the paragraph after Construction \ref{con1} yields the statement.
\end{proof}

In particular, if $t=2$ and $r\le |V(F)|$, then $k=r-1$. Note that for given $t$ and $r$ one can often pick a smaller threshold on $|V(F)|$ than the one given by the proof.

\begin{con}\label{con2}
Let $A_1, A_2,\dots A_{t-1}$ be pairwise disjoint sets of $r-2$ vertices, and let $A$ be the set of $n-(t-1)(r-2)$ vertices not contained in any $A_i$. Let $\cH$ consist of all the $r$-sets of the following form: $A_i\cup\{x,y\}$ where $x,y\in A$.\end{con}

Clearly, for $x,y\in A$, there are exactly $t-1$ hyperedges containing both of them. This means that all the $t$-heavy edges are incident to at least one vertex in an $A_i$. Thus if $F$ does not have a set of vertices of size at most $(t-1)(r-2)$ covering all the edges, then $\cH$ is $\mathbb{H}_t^rF$-free. In particular, if $F$ contains $P_{2(t-1)(r-2)+2}$ or $C_{2(t-1)(r-2)+1}$, then this is the situation. On the other hand, there are $(1-o(1))\binom{n}{2}$ pairs in $A$, there are $t-1$ hyperedges containing each such pair, and they are all distinct hyperedges. This shows $|\cH|=(t-1-o(1))\binom{n}{2}$. Note that this is sharp if $F$ is a long cycle or path by Proposition \ref{tberge} and the known results $ex_r(n,\mathbb{B}_1^rP_k)=O(n)$ \cite{gykl} and $ex_r(n,\mathbb{B}_1^rC_k)=O(n^{1+1/\lfloor k/2\rfloor})$ \cite{gyl}.
This gives the proof of Theorem \ref{pathcycle}.

\newtheorem*{pathcycle}{Theorem \ref{pathcycle}}
\begin{pathcycle}
For any integers $r\ge 3$, $t\ge 2$ and $k\ge 2(t-1)(r-2)+2$ we have
$$ex_r(n,\mathbb{H}_t^rP_k), ex_r(n,\mathbb{B}_t^rP_k)=(t-1-o(1))\binom{n}{2}$$
and
$$ex_r(n,\mathbb{H}_t^rC_{k-1}), ex_r(n,\mathbb{B}_t^rC_{k-1})=(t-1-o(1))\binom{n}{2}.$$
\end{pathcycle}

\section{Cliques}

Mubayi \cite{muba} considered hypergraphs that are 1-heavy copies of $K_k$ and in addition have at most $\binom{k}{2}$ hyperedges. He proved that the largest number of hyperedges in an $r$-uniform hypergraph avoiding each of those is given by the complete balanced $(k-1)$-partite $r$-uniform hypergraph $\cT^r(n,k-1)$ (here balanced means the cardinality of two parts differ by at most one). Observe that $\cT^r(n,k-1)$ does not contain any $\mathbb{H}_1^rK_k$, hence $ex_r(n,\mathbb{H}_1^rK_k)=|\cT^r(n,k-1)|$. Mubayi also proved that $\cT^r(n,k-1)$ asymptotically gives the maximum even if only the $r$-uniform expansion of $K_k$ is forbidden. This result was improved by Pikhurko \cite{pikhu}, who proved that if $n$ is large enough, then $ex_r(n,K_k^{+r})=|\cT^r(n,k-1)|$ holds. This also implies $ex_r(n,\mathbb{B}^r_1K_k)=|\cT^r(n,k-1)|$.

Gerbner, Methuku and Palmer \cite{gmp} moved the threshold down for Berge cliques, using Theorem \ref{celeb2}. The proof relied on the statement that among $K_k$-free blue-red graphs $G$, $\cN(K_r,G_{blue})+|E(G_{red})|$ is always maximized by either a monoblue or a monored graph.

First we extend this theorem.

\begin{theorem}\label{szimm} If $G$ is a $K_k$-free blue-red graph, then $$\cN(K_r,G_{blue})+t|E(G_{red})|+(t-1)\left(\binom{n}{2}-|E(G)|\right)=:g_{r,t}(G)\le$$
$$\le (t-1)\binom{n}{2}+\max\{ex(n,K_r,K_k)-(t-1)ex(n,K_k), ex(n,K_k)\}.$$

\end{theorem}

Observe that, as Zykov proved \cite{zykov} that $ex(n,K_r,K_k)$ is attained by the Tur\'an graph $\cT^2(n,k-1)$ for all $2\le r \le k-1$, the two upper bounds given in Theorem \ref{szimm} are sharp as shown by the monocolored Tur\'an-graphs. It is easy to see that if $n$ is large enough, then the $g_{r,t}$-value of the monoblue Tur\'an graph is larger than that of the monored Tur\'an graph.

The proof of the above theorem is based on the same basic idea as the proof of the case $t=1$ in \cite{gmp} (Theorem 19). It is an adaptation of Zykov's symmetrization process.

\begin{proof}
Let $G$ be a $K_k$-free blue-red graph with the largest value of $g_{r,t}(G)$. Let the \textit{red degree} of $v$ be the number of red edges incident to $v$ and the \textit{blue $r$-clique degree} of $v$ be the number of blue $r$ cliques incident to $v$. Then let
$d^*(v)$ denote the
blue $r$-clique degree plus $t$ times the red degree minus $t-1$ times the degree of $v$ (note that this is the only place where $t$ plays a role). We introduce out first operation on $G$. For two unconnected vertices $u$ and $v$, we delete all the edges incident to $u$ and for every edge $vw$ we add the edge $uw$ of the same color. We call this the \emph{symmetrization} of $u$ to $v$. It is easy to see that the resulting graph is $K_k$-free. If $d^*(u)\le d^*(v)$, then $g_{r,t}(G)$ does not decrease, and if $d^*(u)< d^*(v)$, then $g_{r,t}(G)$ increases.

We will apply several such symmetrization steps. In the first phase we choose a vertex $v$ with the largest $d^*(v)$. Then we pick a vertex $u$ not connected to $v$ and symmetrize $u$ to $v$. Then we repeat this for every vertex not connected to $v$. Observe that $d^*(w)$ may increase for a vertex $w$ during these symmetrization steps, but only if $w$ is connected to $v$. Hence we can symmetrize to $v$ all the vertices not adjacent to $v$, one by one, without decreasing $g_{r,t}(G)$. After this is done, we obtain an independent set $A$ of vertices such that each vertex of $A$ is connected to each vertex $w$ not in $A$. Moreover, the vertices in $A$ are connected to $w$ by edges of the same color. Observe that this property does not change in further symmetrization steps.

In the second phase we pick a vertex not in $A$ and do the same what we did in the first phase. Then we obtain another independent set, and so on. After at most $k-1$ phases we run out of vertices.
At that point we obtained a complete multipartite graph with at most $k-1$ parts such that for any two of its parts, all the edges between the parts are of the same color. Thus the vertices inside a part have the same $d^*$-value. For a part $A$, let $d^*(A)=|A|d^*(a)$ for an $a\in A$. We apply another symmetrization process; for two parts $A$ and $B$ connected by red edges, in one symmetrization step we delete all the edges incident to $B$, except those between $A$ and $B$. Then for each other part $C$, we add all the edges between $B$ and $C$ with the color of the edges between $A$ and $C$. In fact we just recolor some edges, hence the term $(t-1)(\binom{n}{2}-|E(G)|)$ does not matter anymore. If $d^*(B)\le d^*(A)$, then $g_{r,t}(G)$ does not decrease. Observe that $d^*(A)$ does not decrease, as no edge incident to a vertex in $A$ was changed, and no blue clique incident to a vertex in $A$ was deleted, as every edge that we changed is incident to a vertex in $B$, which is connected to vertices in $A$ by red edges.

In the first phase of this second part of the proof, we choose part $A$ with the largest $d^*(A)$ and call it the \emph{active} part. Then we pick a part $B$ connected to it by red edges and symmetrize $B$ to $A$. Then we repeat this for every other part $C$ connected to $A$ by red edges, unless $d^*(C)>d^*(A)$ (it is possible, as $d^*(C)$ can increase while we symmetrize $B$ to $A$). In that case we let $C$ be the active part and symmetrize to $C$ parts that are connected to it by red edges, one by one. It can happen again that we find a part with an even larger $d^*$. However, this can happen finitely many times. Indeed, the largest $d^*$-value is at most $n(\binom{n}{r}+t\binom{n}{2})$, as there are at most $\binom{n}{r}$ blue $r$-cliques, at most $\binom{n}{2}$ red edges and a part can have at most $n$ vertices. If every other part that is connected to the active part with red edges is symmetrized to it, then the first phase ends. Thus the first phase ends after at most $n(k-2)(\binom{n}{r}+t\binom{n}{2})$ symmetrization steps, and we obtain a family $\cA$ of parts connected to each other by red edges and connected to the other parts by blue edges. Then we pick a part not from $\cA$ as the active part and repeat this procedure. Observe that edges incident to parts in $\cA$ do not change later. After at most $k-2$ phases we run out of parts.

At that point we obtain a complete multipartite graph $G'$ such that for any two of its parts, all the edges between the parts are of the same color, and being connected by red edges is an equivalence relation. That means that the $G'_{blue}$ itself is a complete multipartite graph.

Recall that $G'$ has the largest $g_{r,t}$ value among blue-red $K_k$-free graphs (as $g_{r,t}(G')\ge g_{r,t}(G)$). If $G'$ is monochromatic, we are done, thus assume not, and
there are red edges between part $A$ and $B$. If there are no blue edges incident to $A$, then all the edges have to be red and we are done. If vertices of $A$ are not in blue $r$-cliques, then we could recolor the incident blue edges to red, increasing  $g_{r,t}$, a contradiction. Hence we obtained that $A$ is connected to $s\ge r-1$ parts $A_1,\dots, A_s$ in $G'_{blue}$. Let $a_i=|A_i|$ for $i\le s$.

We are going to show that either recoloring the edges between $A$ and $B$ to blue, or recoloring the blue edges incident to $A$ to red increases $g_{r,t}$, leading to a contradiction. Indeed, let $$x=\sum_{1\le i_1<\dots<i_{r-2}\le s} a_{i_1}\cdots a_{i_{r-2}},\hskip 0.5truecm y=\sum_{1\le i_1<\dots<i_{r-1}\le s} a_{i_1}\cdots a_{i_{r-1}},\hskip 0.5truecm z=\sum_{i=1}^sa_i.$$

Then the first operation adds $|A||B|x$ new blue $r$-cliques while deletes $|A||B|$ red edges. The second operation adds $|A|z$ new red edges while deletes $|A|y$ blue $r$-cliques.
If $|A||B|x>t|A||B|$ or $t|A|z>|A|y$, we are done. Otherwise $x\le t$, thus
$x|A|z\le t|A|z \le |A|y$, i.e. $xz\le y$. It is easy to see that every term of $y$ appears in $xz$ and there are additional terms, showing the contradiction.

Hence we obtained that $g_{r,t}$ is maximized by either a monored (complete multipartite) graph, or a monoblue complete multipartite graph, in which case $x\ge t$. As in a monored graph $g_{r,t}$ is the number of edges, clearly it is maximized by the Tur\'an graph, using Tur\'an's theorem. Let us assume now that $g_{r,t}$ is maximized by a monoblue complete multipartite graph on $n$ vertices which is not the Tur\'an graph, i.e. there are parts $A$ and $B$ such that $|A|\ge |B|+2$. We define the $a_i$'s and $x$ as above. Then we move a vertex from $A$ to $B$. The resulting graph has more edges by $(|A|-1)(|B|+1)-|A||B|>0$ and more copies of $K_r$ by $x(|A|-1)(|B|+1)-|A||B|\ge t(|A|-1)(|B|+1)-|A||B|$. This implies $g_{r,t}$ increases, a contradiction.
\end{proof}

Together with Theorem \ref{celebrated2} this implies the following.

\begin{corollary}\label{kompl}
 $ex_r(n,\mathbb{B}_t^rK_k)\le (t-1)\binom{n}{2}+\max\{ex(n,K_r,K_k)-(t-1)ex(n,K_k),ex(n,K_k)\}$.
\end{corollary}

Let us recall that $\cQ^r_t(n,k-1)$ is obtained from $\cT^r(n,k-1)$ by adding for every  pair of vertices $u$ and $v$ from the same part, $t-1$ distinct hyperedges that contains them and intersects $r-2$ other parts.
Observe that $\cT^r(n,k-1)$ has $ex(n,K_r,K_k)$ hyperedges by the theorem of Zykov and we added $(t-1)(\binom{n}{2}-|E(\cT^2(n,k-1))|)=(t-1)(\binom{n}{2}-ex(n,K_k))$ additional hyperedges. On the other hand $\cQ^r_t(n,k-1)$ is $\mathbb{H}_t^rK_k$-free, as the edges contained in at least $t$ hyperedges form a $\cT^2(n,k-1)$.

\newtheorem*{klikk}{Theorem \ref{klikk}}
\begin{klikk}
Let $k>r>2$, $t\ge 1$ and $n$ be large enough. Then \[ex_r(n,\mathbb{B}_t^rK_k)=ex_r(n,\mathbb{H}_t^rK_k)=|\cQ^r_t(n,k-1)|.\]
\end{klikk}

\begin{proof}
 As $n$ is large enough, $\cQ^r_t(n,k-1)$ is defined. On the other hand, having Corollary \ref{kompl} in hand, it is enough to show that $ex(n,K_r,K_k)\ge t\cdot ex(n,K_k)$ if $n$ is large enough. It is obvious as the order of magnitude of $ex(n,K_r,K_k)$ is $n^r$, while $ex(n,K_k)$ is quadratic.
\end{proof}

Note that using Corollary \ref{kompl} one could also determine a reasonable threshold for $n$ with some computation.

\section{The 3-uniform case}

In this section, we consider results for the smallest uniformity, i.e. in the case $r=3$.

\subsection{$t=2$}

In this case, \textbf{(i)} of Corollary \ref{coroll} gives  the lower bound $(1-o(1))n^2/6\le ex_3(n,\mathbb{H}_2^3F)$ for every graph $F$.

Let $F$ be a connected graph, which is not a single edge. Construction \ref{con1} gives the lower bound $(1-o(1))n^2/4\le ex_3(n,\mathbb{H}_2^3F)$. Let us repeat the construction in this special case: we take a matching $e_1,\dots, e_{\lfloor n/2\rfloor}$ and add every triple of the form $\{v\}\cup e_i$ if $v\in e_j$ with $j<i$. This has cardinality $(1-o(1))n^2/4$.

Construction \ref{con2} gives the lower bound $(1-o(1))n^2/2\le ex_3(n,\mathbb{H}_2^3F)$ for every connected graph $F$ that is not a star. The construction in this case consists of all triples that contain a fixed vertex $v$.

\begin{lemma}\label{csillag}
 For any positive integer $r\ge 2$ we have $$ex_3(n,\mathbb{B}_2^3S_r)\le (1+o(1))\frac{n^2}{4}.$$
\end{lemma}

\begin{proof}
We start with the following observation.

\begin{proposition}\label{fat}
If $\cF\subseteq \binom{[n]}{3}$ is $\mathbb{B}_2^3S_r$-free, then the graph of $2$-heavy edges does not contain an $S_{3r}$.
\end{proposition}
\begin{proof}
Suppose to the contrary that the degree of $x$ in the graph of 2-heavy edges is at least $3r$ and let $y_1,y_2,\dots,y_{3r}$ be $3r$ of $x$'s neighbors. We now greedily define our copy of $\mathbb{B}_2^3S_r$. As $xy_1$ is contained in at least 2 hyperedges $H^1_1,H^1_2$, we take these two hyperedges and delete the third vertices of $H^1_1,H^1_2$ from the list of $y_i$'s. At every step, we add two hyperedges to our copy of $\mathbb{B}_2^3S_r$ and delete at most 3 vertices (including the current $y_i$) from our list. Therefore, by our last step we deleted at most $3(r-1)$ $y_i$'s, so we can still define $H^r_1,H^r_2$ to complete the copy of $\mathbb{B}_2^3S_r$.
\end{proof}
Proposition \ref{fat} implies that the number of 2-heavy edges is at most $\frac{3}{2}rn$. Let us call a hyperedge $H$ of a $\mathbb{B}_2^3S_r$-free 3-graph $\cF$ \textit{bad} if it contains more than one 2-heavy edges. We claim that every 2-heavy edge is contained in at most $6r$ bad hyperedges. Indeed, if $H_1,H_2,\dots, H_{6r+1}$ are bad hyperedges containing $uv$ and $x_1,x_2,\dots,x_{6r+1}$ are their third vertices, then $$\max\{|\{j:x_ju ~\text{is 2-heavy}\ \}|,|\{j:x_jv ~\text{is 2-heavy}\ \}|\}\ge 3r,$$ so the graph of 2-heavy edges would contain $S_{3r}$ contradicting Proposition \ref{fat}. We obtained that the number of bad hyperedges is at most $9r^2n$. On the other hand every non-bad hyperedge $H$ contains at least two non-2-heavy edges, and these edges, by definition, are only contained in one hyperedge of $\cF$, namely $H$. Therefore the number of non-bad hyperedges is at most $\frac{1}{2}\binom{n}{2}=(1+o(1))n^2/4$.
\end{proof}

In case $F$ is not a star, we have an upper bound $ex_3(n,\mathbb{B}^3_2F)\le ex_3(n,\mathbb{B}^r_{1}F)+\binom{n}{2}$ by Proposition \ref{tberge}. In case $ex_3(n,\mathbb{B}^3_{1}F)$ is sub-quadratic, this matches asymptotically the lower bound given by Construction \ref{con2}.

This means that for any connected $F$, if we know $ex_3(n,\mathbb{B}_1^3F)$, and it is sub-quadratic or super-quadratic, then we know the asymptotics of $ex_3(n,\mathbb{B}^3_2F)$. But if $ex_3(n,\mathbb{B}_1^3F)$ is quadratic, we do not know the asymptotics of $ex_3(n,\mathbb{B}^3_2F)$. One such example is the triangle. Gy\H ori \cite{gyori} proved $ex_3(n,\mathbb{B}_1^3K_3)=(1+o(1))n^2/8$ (note that he proved an exact bound for many $n$).

Construction \ref{con2} gives $ex_3(n,\mathbb{H}_2^3K_3)\ge (1+o(1))n^2/2$. It is easy to see that it is sharp: by definition every hyperedge $H$ in a $\mathbb{H}_2^3K_3$-free hypergraph  contains an edge that is not 2-heavy, i.e. that is contained only in $H$ and no other hyperedges. As there are $\binom{n}{2}$ edges that can play that role, there are at most $\binom{n}{2}$ hyperedges. We show that the same asymptotic result holds for 2-wise Berge triangles.

\begin{proposition} $ex_3(n,\mathbb{B}_2^3K_3)= (1+o(1))n^2/2$.

\end{proposition}

\begin{proof} Let us consider a 3-uniform, $\mathbb{B}_2^3K_3$-free hypergraph $\cH$. If it contains edges that are contained in exactly one hyperedge, let us delete a hyperedge containing such an edge. Then we continue this as long as we can. Let $\cH'$ be the resulting hypergraph. Then every hyperedge of $\cH'$ contains three 2-heavy edges. Let us call an edge \textit{nice} if it is contained in exactly two hyperedges.

We claim that every hyperedge $H$ of $\cH'$ contains at least two nice edges. Indeed, if one edge $e_1 \subset H$ is also contained in $H_1$ and the other two edges $e_2,e_3\subset H $are contained in at least three hyperedges (so $e_2 \subset H,H_2,H'_2, e_3\subset H,H_3,H_3'$), then these six hyperedges form a 2-wise Berge triangle.

We claim that for some $k$ we can find $k$ hyperedges such that there are at least $k$ edges contained only in them. Indeed, let us start with an arbitrary hyperedge $H$, and put it in a hypergraph $\cG$. It contains two nice edges $e_1$ and $e_2$. Put them in a graph $G$. Let $e_1$ be contained in $H_1$ and $e_2$ in $H_2$ in addition to $H$, then both $H_1$ and $H_2$ contain another nice edge which are contained in other hyperedges and so on. We always put the hyperedges we reach this way to $\cG$ and the nice edges to $G$. At some point, we have to stop, $\cG$ and $G$ does not increase further. At this point the edges of $G$ are contained only in hyperedges of $\cG$, and each are contained in two hyperedges of $\cG$. On the other hand, each hyperedge of $\cG$ contains at least two edges of $G$. This means $k:=|\cG|\le |E(G)|$.

Let us now delete $\cG$ and find another set $\cG'$ of hyperedges with more edges contained only in them, delete $\cG'$, and so on. This procedure stops only when all the hyperedges are deleted. Observe that whenever we deleted some hyperedges, we deleted at least that many edges from the shadow graph (the edges that were contained only in those hyperedges) Indeed, when obtaining $\cH'$, we always delete a hyperedge and at least one edge of the shadow graph with it, and later we always delete $k$ hyperedges and at least $k$ edges of the shadow graph. This implies that originally the number of hyperedges was at most the number of edges in the shadow graph, which is at most $\binom{n}{2}$, finishing the proof.
\end{proof}

\subsection{$t>2$}

Let us repeat first Construction \ref{con2} in this special case. We take $t-1$ fixed vertices $v_1,\dots,v_{t-1}$ and every triple that contains exactly one of them. The resulting hypergraph is $\mathbb{H}_t^3F$-free if $F$ is not a subgraph of $K_{t-1,n-t+1}$. This gives the lower bound $ex_3(n,\mathbb{H}_t^3F)\ge (1+o(1))(t-1)n^2/2$ hyperedges, and this bound is asymptotically sharp if additionally $ex_3(n,\mathbb{B}^3_1F)=o(n^2)$.

Let us turn our attention to Construction \ref{con1}. As we mentioned, it is not hard to improve that construction for some specific $t$ and $r$. Here by improvement we do not mean a larger hypergraph, but a larger set of graphs $F$ for which the hypergraph avoids $t$-heavy copies of $F$.

Our first goal is to find regular graphs not containing $F$, preferably ones which contain a perfect matching (we will see later why).
Let $F$ be a connected graph that contains a cycle $C_l$. It is well-known that there exist $d$-regular graphs with girth at least $l+1$ for any $d$. Moreover, a result of Sachs (Theorem 1 in \cite{sachs}) states that there exist such graphs that contain a perfect matching.

If $F$ is a tree and has $k$ vertices, then any $d$-regular graph with $d\ge k-1$ contains $F$. If $d\le k-2$ is odd, then $K_{d+1}$ is an $F$-free $d$-regular graph, that also has a perfect matching. If $d\le k-3$ is even, then $K_{d+2}$ minus a perfect matching is an $F$-free $d$-regular graph that also has a perfect matching. Finally, if $d=k-2$ even, then for
$S_{k-1}$ there exists a $d$-regular $S_{k-1}$-free graph that has a perfect matching (actually, any $d$-regular graph with a perfect matching is $S_{k-1}$-free), while for other trees, it is shown in \cite{gptv} that
the only $d$-regular $P_k$-free graphs are disjoint unions of $K_{d+1}$'s and as $d+1$ is odd, they do not contain perfect matchings.

\begin{con}\label{con3}

Let $F$ be connected and $t$ odd.
Let $m$ be a number such that there exists a $((t-1)/2)$-regular $F$-free graph $G$ on $m$ vertices. Then we take $\lfloor n/m \rfloor$ disjoint sets $A_i$ of size $m$. For every $i$, we take a copy of $G$ on $A_i$. For every $i<j$, we take every triple that consists of two vertices in $A_i$, connected by an edge in $G$, and one vertex from $A_j$, and we also take every triple that consists of two vertices in $A_j$, connected by an edge in $G$, and one vertex from $A_i$. Let $\cH$ be the hypergraph obtained this way. Then for every $u\in A_i$ and $j\in A_j$ there are exactly $t-1$ hyperedges containing both $u$ and $v$. This means that the $t$-heavy edges form $\lfloor n/m\rfloor$ vertex disjoint copies of $G$, thus $\cH$ is $\mathbb{H}_t^3F$-free.

Let us consider now the case $t$ is even. Then
let $m$ be a number such that there exists a $t/2$-regular $F$-free graph $G$ on $m$ vertices that has a matching $M$. Then we proceed similarly, just we avoid the edges in $M$ from $A_j$. More precisely, we
take $\lfloor n/m \rfloor$ disjoint sets $A_i$ of size $m$. For every $i$, we take a copy of $G$ on $A_i$. For every $i<j$, we take every triple that consists of two vertices in $A_i$, connected by an edge in $G$, and one vertex from $A_j$, and we also take every triple that consists of two vertices in $A_j$, connected by an edge in $G$ that is not in $M$, and one vertex from $A_i$. Then, again, for every $u\in A_i$ and $j\in A_j$ there are exactly $t-1$ hyperedges containing both $u$ and $v$. This means that the $t$-heavy edges form $\lfloor n/m\rfloor$ vertex disjoint copies of $G$, thus $\cH$ is $\mathbb{H}_t^3F$-free.
\end{con}

In both these cases $\cH$ has $(1+o(1))(t-1)n^2/4$ hyperedges. Indeed, there are $(1+o(1))\binom{n}{2}$ edges $uv$ with $u\in A_i$, $v\in A_j$, $i\neq j$, and they are contained in exactly $t-1$ hyperedges. However, we count every hyperedge twice this way, as every hyperedge contains two vertices from the same $A_i$.

Hence $ex_3(n,\mathbb{H}_t^3F)\ge (1+o(1))(t-1)n^2/4$ if there exists a $(t-1)/2$-regular $F$-free graph, or a $t/2$-regular $F$-free graph that has a perfect matching. In particular, it holds if $F$ contains a cycle, or $t$ is not divisible by 4 and $\lfloor t/2\rfloor\le |V(F)|-2$, or $t$ is divisible by 4 and $t/2\le |V(F)|-3$.

For many graphs $F$, $ex_3(n,\mathbb{H}_t^3F)$ has the lower bound $(1+o(1))(t-1)n^2/2$ by Construction \ref{con2}, and for almost all graphs we have the lower bound $(1+o(1))(t-1)n^2/4$ by Construction \ref{con3}. The missing graphs are trees on at most $\lceil (t-1)/2\rceil+2$ vertices. For them we have the lower bound $(1+o(1))(t-1)n^2/6$ by Corollary \ref{quad}. Observe that we can improve this by combining the above constructions.

\begin{con}\label{con4}
Let $t_0\le n/2$ be the largest integer such that $F$ is not a subgraph of $K_{t_0-1,n-t_0+1}$ and $k$ be the number of vertices of $F$. Note that if $n$ is large, we have $t_0<k$, as $K_{k-1,n-k+1}$ contains every tree on $k$ vertices. Let $\cH_1$ be the $\mathbb{H}_t^3F$-free hypergraph given by Construction \ref{con2}, i.e. we fix $t_0$ vertices and take every hyperedge that contains exactly one of them. Let $\cH_2$ be the $\mathbb{H}_{2k-3}^3F$-free hypergraph given by Construction \ref{con3} with $m=k-1$ and $G$ being $K_{k-1}$, i.e. we take $\lfloor n/(k-1)\rfloor$ sets $A_i$ and every hyperedge that contains two vertices from an $A_i$ and one vertex from another.

Let us assume $t>t_0$ and $t\ge 2k-3$. Then we take a family $\cF_1$ of cardinality $(1+o(1))(t-t_0)n^2/6$ where every edge is contained in at most $t-t_0$ hyperedges and a family $\cF_2$ of cardinality $(1+o(1))(t-2k+3)n^2/6$ where every edge is contained in at most $t-2k+3$ hyperedges. We claim that $\cH_i\cup \cF_i$ is $\mathbb{H}_t^3F$-free for $i=1,2$. Indeed, $\cH_i$ is $\mathbb{H}_t^3F$-free by definition, and $\cF_i$ does not change which edges are $t$-heavy.
\end{con}

Let us examine the size of $\cH_i\cup \cF_i$. We have $|\cH_1|=(1+o(1))(t_0-1)n^2/2$ and $|\cH_2|=(1+o(1))(2k-4)n^2/4$. We claim that $|\cH_1\cup \cF_1|=(1+o(1))(t_0-1)n^2/2+(1+o(1))(t-t_0)n^2/6$ and $|\cH_2\cup \cF_2|=(1+o(1))(2k-4)n^2/4+(1+o(1))(t-2k+3)n^2/6$. To see this, we need to show that the intersection of $\cH_i$ and $\cF_i$ is sub-quadratic for $i=1,2$. Observe that there is a set $S_i$ of linearly many edges such that each hyperedge in $\cH_i$ contains at least one member of $S_i$ for $i=1,2$. On the other hand, $\cF_1$ has at most $(t-t_0)|S_1|$ hyperedges containing a member of $S_1$ and $\cF_2$ has at most $(t-2k+3)|S_2|$ hyperedges containing a member of $S_2$.

Let us consider now the simplest graph where we do not have an asymptotic result, the cherry $S_2$. We have $t_0=1$, thus $\cH_1$ is empty. On the other hand $k=3$, thus $\cH_2$ is not empty, the $A_i$'s can have two vertices.

\begin{proposition}\label{uj} We have

(i) $ex_3(n,\mathbb{H}_2^3S_2)=(1+o(1))n^2/4$, and

(ii) for $t\ge 3$, $ex_3(n,\mathbb{H}_t^3S_2)\ge (t+o(1))n^2/6$.

\end{proposition}

\begin{proof}

The first statement is true by Construction \ref{con1} and Lemma \ref{csillag}. The second statement is true by  Construction \ref{con3} for $t=3$ and by Construction \ref{con4} for larger $t$.

\end{proof}

We can prove that the second bound is asymptotically sharp, but only for the $t$-heavy cherries and not the $t$-wise Berge cherries.

\begin{proposition} For $t\ge 3$ we have $ex_3(n,\mathbb{H}_t^3S_2)=(t+o(1))n^2/6$.

\end{proposition}

\begin{proof} The lower bound is given by Proposition \ref{uj}. For the upper bound let us consider a $\mathbb{H}_t^3S_2$-free hypergraph $\cH$, and let $S$ be the set of $t$-heavy edges. Then $S$ must be a matching. Observe that for every edge $e$ not in $S$ there are at most two hyperedges that contain $e$ and a member of $S$. Let $m$ be the number of hyperedges that contain a member of $S$. They each contain two edges not from $S$, and this way we count such edges at most twice, thus we have $m\le \binom{n}{2}$.

Let us add up the number of times the edges not in $S$ are contained in hyperedges. On the one hand it is at most $(t-1)\binom{n}{2}$. On the other hand it is exactly $2m+3(|\cH|-m)$. Thus we have
\[3|\cH|-\binom{n}{2}\le 3|\cH|-m\le (t-1)\binom{n}{2},\]
and simple rearranging finishes the proof.
\end{proof}

\textbf{Acknowledgement} We are most grateful to the anonymous referee for reading the manuscript carefully and providing useful comments.

\textbf{Funding}: Research supported by the National Research, Development and Innovation Office - NKFIH under the grants SNN 129364, KH 130371 and K 116769, by the J\'anos Bolyai Research Fellowship of the Hungarian Academy of Sciences and the Taiwanese-Hungarian Mobility Program of the Hungarian Academy of Sciences, by Ministry of Science and Technology Project-based Personnel Exchange Program 107 -2911-I-005 -505.

\end{document}